\documentclass[11pt
]{article}

\usepackage[all]{xy}

\usepackage[latin1]{inputenc}
\usepackage{amsfonts}
\usepackage{amsmath}
\usepackage{amssymb}
\usepackage{amscd}
\usepackage{amsthm}
\usepackage{indentfirst}
\usepackage{stmaryrd}
\usepackage[hmargin=3cm
,vmargin=3cm
]{geometry}

\usepackage{epsfig}

\newtheorem{thm}{Theorem}[section]
\newtheorem{lemma}[thm]{Lemma}
\newtheorem{prop}[thm]{Proposition}
\newtheorem{coroll}[thm]{Corollary}

\theoremstyle{definition}

\newtheorem{rem}[thm]{Remark}

\newtheorem*{acknow}{Acknowledgements}

\newcommand{\R}{{\mathbb{R}}}

\newcommand{\T}{{\mathbb{T}}}
\newcommand{\Z}{{\mathbb{Z}}}
\newcommand{\N}{{\mathbb{N}}}
\newcommand{\C}{{\mathbb{C}}}
\newcommand{\HH}{{\mathbb{H}}}
\newcommand{\K}{{\mathbb{K}}}

\newcommand{\cL}{{\mathcal{L}}}

\newcommand{\mft}{{\mathfrak{t}}}
\newcommand{\mfk}{{\mathfrak{k}}}
\newcommand{\mfh}{{\mathfrak{h}}}

\newcommand{\mfg}{{\mathfrak{g}}}
\newcommand{\mfu}{{\mathfrak{u}}}
\newcommand{\mfsu}{{\mathfrak{su}}}
\newcommand{\mfpu}{{\mathfrak{pu}}}

\newcommand{\fc}{{:\ }}

\newcommand{\ol}{\overline}
\newcommand{\wt}{\widetilde}
\newcommand{\wh}{\widehat}

\DeclareMathOperator{\id}{id}

\DeclareMathOperator{\Gr}{G}
\DeclareMathOperator{\PU}{PU}
\DeclareMathOperator{\Uni}{U}
\DeclareMathOperator{\SU}{SU}
\DeclareMathOperator{\SO}{SO}
\DeclareMathOperator{\PGL}{PGL}
\DeclareMathOperator{\Stie}{V}
\DeclareMathOperator{\PV}{PV}

\DeclareMathOperator{\pr}{pr}

\DeclareMathOperator{\pt}{pt}

\DeclareMathOperator{\tr}{tr}

\DeclareMathOperator{\supp}{supp}

\DeclareMathOperator{\Ham}{Ham}

\DeclareMathOperator{\Cont}{Cont}
\DeclareMathOperator{\Diff}{Diff}

\DeclareMathOperator{\Lie}{Lie}

\DeclareMathOperator{\diag}{diag}

\DeclareMathOperator{\std}{std}

\title
{Quasi-morphisms on contactomorphism groups and Grassmannians of $2$-planes}
\author{Frol Zapolsky}
\date{}

\begin{document}

\renewcommand{\labelenumi}{(\roman{enumi})}

\maketitle

\begin{abstract}We construct a natural prequantization space over a monotone product of a toric manifold and an arbitrary number of complex Grassmannians of $2$-planes in even-dimensional complex spaces, and prove that the universal cover of the identity component of the contactomorphism group of its total space carries a nonzero homogeneous quasi-morphism. The construction uses Givental's nonlinear Maslov index and a reduction theorem for quasi-morphisms on contactomorphism groups previously established together with M.\ S.\ Borman. We explore applications to metrics on this group and to symplectic and contact rigidity. In particular we obtain a new proof that the quaternionic projective space $\HH P^{n-1}$, naturally embedded in the Grassmannian $\Gr_2(\C^{2n})$ as a Lagrangian, cannot be displaced from the real part $\Gr_2(\R^{2n})$ by a Hamiltonian isotopy.
\end{abstract}

\section{Introduction and results}

This paper is about quasi-morphisms on contactomorphism groups and their applications. We construct a natural prequantization space $V$ over a monotone symplectic manifold of the form
$$M = N \times \prod_{i=1}^{r}\Gr_2(\C^{2n_i})$$
for $n_1,\dots,n_r \geq 2$. Here $\Gr_k(\C^n)$ is the complex Grassmannian of $k$-planes in $\C^n$ while $N$ is a closed symplectic toric manifold which is even, a condition on it moment polytope defined below. This $V$ carries a natural contact structure, and the main result is the existence of a quasi-morphism on $\wt\Cont_0(V)$, the universal cover of the identity component of its contactomorphism group. The quasi-morphisms we construct have some interesting and useful properties. As a consequence we deduce the orderability of $V$, as well as the existence and unboundedness of natural metrics on $\wt\Cont_0(V)$. We also establish rigidity results for subsets of $V$ and $M$.

\subsection{Basic notions and the results}

Let us now pass to precise definitions and formulations. All the material presented here can be found in \cite[Section 1.1]{Borman_Zapolsky_Qms_Cont_rigidity}. A homogeneous quasi-morphism on a group $\Gamma$ is a function $\mu \fc \Gamma \to \R$ satisfying
\begin{equation}\label{eqn:defin_qm}
\sup_{a,b\in\Gamma}|\mu(ab) - \mu(a) - \mu(b)| < \infty\quad\text{and}\quad\mu(a^k) = k\mu(a)\text{ for }a\in \Gamma\,, k \in \Z\,.
\end{equation}
In this paper by a quasi-morphism we mean a nonzero homogeneous quasi-morphism. See \cite{Calegari_scl} for basics on quasi-morphisms.

A contact manifold is a pair $(V,\xi)$ where $V$ is an odd-dimensional manifold while $\xi$ is a contact structure on $V$, that is a maximally nonintegrable hyperplane distribution. For us $V$ is always closed and connected, while $\xi$ is cooriented, meaning that the line bundle $TV/\xi$ is given an orientation. The orientability of this line bundle is equivalent to the existence of a $1$-form $\alpha$ such that $\xi = \ker \alpha$. Such $\alpha$ is a contact form, that is $\alpha\wedge d\alpha^{\frac 1 2(\dim V - 1)}$ is a volume form. We let $\Cont_0(V,\xi)$ be the identity component of the group of contactomorphisms of $(V,\xi)$, that is diffeomorphisms preserving $\xi$. We denote by $\wt\Cont_0(V,\xi)$ the universal cover of $\Cont_0(V,\xi)$. See \cite{Geiges_Intro_ct_topology} for foundations of contact geometry.

The main example of contact manifolds for us is given by prequantization bundles, also known as Boothby--Wang manifolds \cite{Boothby_Wang_Contact_mfds}. A prequantization bundle (or space) is a quintuple $(V,\alpha,\pi,M,\omega)$ where $(M,\omega)$ is a symplectic manifold, $\pi\fc V \to M$ is a principal $S^1$-bundle with an $S^1$-invariant connection form $\alpha \in \Omega^1(V)$ satisfying $d\alpha = \pi^*\omega$. In this case $\alpha$ is a contact form, which means that $V$ acquires a natural cooriented contact structure $\xi = \ker\alpha$. Whenever we speak about prequantization spaces, we assume that the total space is a contact manifold with this cooriented contact structure.

The quasi-morphisms we construct are linked to a natural binary relation on $\wt\Cont_0(V,\xi)$, due to Eliashberg--Polterovich \cite{Eliashberg_Polterovich_Partially_ordered_grps_geom_cont_transfs}. Namely, we write $\wt\phi \preceq \wt\psi$ for $\wt\phi,\wt\psi \in \wt\Cont_0(V,\xi)$ if $\wt\phi^{-1}\wt\psi$ can be represented by a contact isotopy starting at $\id_V$ whose generating contact vector field is everywhere nonnegatively directed relative to $\xi$, which makes sense since $\xi$ is cooriented. One easily checks that $\preceq$ is a reflexive and transitive biinvariant relation. If it is also symmetric, it is a partial order, and then $(V,\xi)$ is said to be orderable \cite{Eliashberg_Polterovich_Partially_ordered_grps_geom_cont_transfs}, \cite{Eliashberg_Kim_Polterovich_Geom_cont_transfs_dom_orderability}. We call a quasi-morphism $\mu \fc \wt\Cont_0(V,\xi) \to \R$ monotone if $\mu(\wt\phi) \leq \mu(\wt\psi)$ whenever $\wt\phi \preceq \wt\psi$.

Given a coorienting contact form $\alpha$ for $\xi$, there is a natural bijection between contact isotopies starting at $\id_V$ and smooth functions $h \fc [0,1] \times V \to \R$, called contact Hamiltonians. Namely, a contact isotopy generated by a time-dependent contact vector field $X^t$ corresponds to the Hamiltonian $\alpha(X^t)$. For a contact Hamiltonian $h$, we let $\wt\phi_h\in\wt\Cont_0(V,\xi)$ be the class of the path given by the contact isotopy corresponding to $h$. The contact vector field corresponding to the contact Hamiltonian $1$ is called the Reeb vector field of $\alpha$, and it is denoted by $R_\alpha$. It is determined by the equations $\alpha(R_\alpha) = 1$ and $\iota_{R_\alpha}d\alpha = 0$. Note that if $V$ is the total space of a prequantization bundle and $\alpha$ is the connection form, $R_\alpha$ is the infinitesimal generator of the $S^1$-action.

We call a quasi-morphism $\mu \fc \wt\Cont_0(V,\xi) \to \R$ continuous if $\mu(\wt\phi_{h^{(k)}}) \to \mu(\wt\phi_h)$ whenever $h^{(k)}$ is a sequence of contact Hamiltonians $C^0$ converging to $h$.

We call a subset $S \subset V$ displaceable if there is $\psi \in \Cont_0(V,\xi)$ such that $\psi(S) \cap \ol S = \varnothing$, where $\ol S$ is the closure of $S$. A quasi-morphism $\mu \fc \wt\Cont_0(V,\xi) \to \R$ is said to have the vanishing property if $\mu(\wt\phi_h) = 0$ whenever $h$ is a contact Hamiltonian for which there exists a displaceable subset $U \subset V$ with $\supp h_t \subset U$ for all $t$.

A symplectic toric manifold is a symplectic manifold $(M,\omega)$ endowed with an effective Hamiltonian action of a torus $\T$ of maximal possible dimension $\frac 1 2 \dim M$ \cite{Audin_Torus_actions_sympl_mfds}. Such an action is given in terms of a moment map $\Phi_\T \fc M \to \mft^*$ where $\mft = \Lie (\T)$, meaning the infinitesimal action of $X \in \mft$ is the Hamiltonian vector field of $\langle \Phi_\T(\cdot),X\rangle \in C^\infty(M)$. The image $\Delta:= \Phi_\T(M) \subset \mft^*$ is a polytope whose top-dimensional faces have conormals $\nu_1,\dots,\nu_d$, which are primitive vectors in the integer lattice $\mft_\Z:=\ker(\exp\fc \mft \to \T)$. We refer the reader to \cite{Audin_Torus_actions_sympl_mfds} for background on toric manifolds.

We call $\Delta$ and $M$ even if
$$\sum_{j=1}^{d}\nu_j \in 2\mft_\Z \subset \mft_\Z\,.$$
This notion was introduced in \cite{Borman_Zapolsky_Qms_Cont_rigidity}. See Section 1.6 \emph{ibid.}\ for examples of even and non-even toric manifolds. Finally, we say that a symplectic manifold $(M,\omega)$ is monotone if its first Chern class and $\omega$ are positively proportional as functionals on $H_2(M;\Z)$.

Here is the main result, whose proof occupies Sections \ref{ss:contruction_preq_bdle}, \ref{ss:constructing_qm}.
\begin{thm}\label{thm:main_result}Given a closed connected even monotone symplectic toric manifold $N$ and integers $r \geq 0$, $n_1,\dots,n_r \geq 2$, there is a prequantization bundle with total space $V$ and contact structure $\xi$ over the monotone product
$$N\times\prod_{i=1}^{r}\Gr_2(\C^{2n_i})\,,$$
and a monotone continuous quasi-morphism $\mu\fc\wt\Cont_0(V,\xi) \to \R$ having the vanishing property.
\end{thm}

\begin{rem}\label{rem:after_main_result}
\begin{enumerate}
  \item The case $r=0$ is proved in \cite{Borman_Zapolsky_Qms_Cont_rigidity} and is included here for completeness. The present theorem may be viewed as an extension of the results presented \emph{ibid.}
  \item The prequantization bundle is naturally constructed from the moment polytope of $N$.
  \item The specific form of the base is dictated by the necessity to prove rigidity of certain subsets participating in the construction of $\mu$, and cannot be improved, for instance to more general Grassmannians, using the methods presented here. See Section \ref{ss:constructing_qm} for a more detailed explanation. Briefly, we must use the fact that $\C^{2n_i}$ is identifiable with the quaternionic space $\HH^{n_i}$ together with its module structure over $\HH$.
\end{enumerate}
\end{rem}

The mere existence of such a quasi-morphism has numerous consequences. First, \cite[Theorem 1.28]{Borman_Zapolsky_Qms_Cont_rigidity}, based on the results in \cite{Eliashberg_Polterovich_Partially_ordered_grps_geom_cont_transfs}, implies the following fact.
\begin{coroll}\label{coroll:preq_bdles_orderable}
The contact manifold $(V,\xi)$ in Theorem \ref{thm:main_result} is orderable. \qed
\end{coroll}
\noindent It is a direct consequence of the existence of a nonzero monotone homogeneous quasi-morphism on $\wt\Cont_0(V,\xi)$.

Next we consider metrics on $\wt\Cont_0(V,\xi)$. Following \cite{Burago_Ivanov_Polterovich_Conj_invt_norms_grps_geom_origin}, if $G$ is a group, a norm on $G$ is a function $|\cdot|\fc G \to \R_{\geq 0}$ such that $|g| = 0$ if and only if $g = 1$, $|gh| \leq |g| + |h|$, $|g^{-1}| = |g|$, for all $g,h \in G$. An important class is formed by so-called stably unbounded norms, where a norm $|\cdot|$ is stably unbounded if there is $g \in G$ such that $\lim_{k\to\infty}\frac{|g^k|}{k} > 0$. A norm is conjugation-invariant if in addition $|ghg^{-1}|=|h|$ for $g,h\in G$. Norms on $G$ are in bijection with right-invariant metrics on $G$: a norm $|\cdot|$ corresponds to the metric $d(g,h)=|gh^{-1}|$, which is biinvariant if and only if $|\cdot|$ is conjugation-invariant. Conjugation-invariant norms play an important role in the study of the geometry of groups, which is especially interesting when $G$ is an infinite-dimensional symmetry group, see \emph{ibid}.

In \cite{Fraser_Polterovich_Rosen_Sandon_type_metrics_Cont} the authors define a conjugation-invariant norm on $\wt\Cont_0(V,\xi)$, where $(V,\xi)$ is the total space of a prequantization bundle, which we refer to as the Fraser--Polterovich--Rosen norm. In \cite{Shelukhin_Hofer_norm_contacto} the author defines a natural family of norms on $\wt\Cont_0(V,\xi)$, which are not conjugation-invariant, for any (closed) contact manifold $(V,\xi)$; we refer to these as the Shelukhin norms. In the context of norms on groups, the quasi-morphism of Theorem \ref{thm:main_result} yields the following result, where for a quasi-morphism $\mu$ its defect is defined to be the supremum in equation \eqref{eqn:defin_qm} above.
\begin{thm}
Let $(V,\xi)$ be the total space of a prequantization bundle as in Theorem \ref{thm:main_result}, let $\mu$ be the corresponding quasi-morphism, and let $D_\mu$ be its defect.
\begin{enumerate}
  \item The function $\wt\Cont_0(V,\xi) \to \R$, $\wt\phi \mapsto |\mu(\wt\phi)|+D_\mu$ is a stably unbounded conjugation-invariant norm.
  \item The Fraser--Polterovich--Rosen norm is stably unbounded on $\wt\Cont_0(V,\xi)$.
  \item The Shelukhin norm is stably unbounded on $\wt\Cont_0(V,\xi)$.
\end{enumerate}
\end{thm}
\begin{proof}
Item (i) follows from the elementary general fact that if $\nu \fc G \to \R$ is a homogeneous quasi-morphism with defect $D_\nu$, then $|\nu| + D_\nu$ is a stably unbounded conjugation-invariant norm. Item (ii) is just \cite[Theorem 2.15]{Fraser_Polterovich_Rosen_Sandon_type_metrics_Cont}, based on the fact that $(V,\xi)$ is orderable by Corollary \ref{coroll:preq_bdles_orderable}. Item (iii) follows from \cite[Proposition 33]{Shelukhin_Hofer_norm_contacto}, based on the existence of the homogeneous monotone quasi-morphism $\mu$.
\end{proof}

Our final application is to symplectic and contact rigidity. Before we formulate the results, it is instructive to see what Theorem \ref{thm:main_result} gives in the case $N = \pt$, $r=1$, that is when the base of the prequantization bundle is $\Gr_2(\C^{2n})$. As we will see in Section \ref{ss:contruction_preq_bdle}, in this case the total space $V$ fits into the diagram
$$S^{8n-1} \supset \Stie_2(\C^{2n})\xrightarrow{\wh\chi}V = \Stie_2(\C^{2n})/\SU(2)\,,$$
where $S^{8n-1} \subset \C^{4n}$ is the sphere of radius $\sqrt 2$, $\Stie_2(\C^{2n}) \subset (\C^{2n})^2 = \C^{4n}$ is the Stiefel variety of Hermitian orthonormal $2$-frames in $\C^{2n}$, $\SU(2)$ acts on it in an obvious way, while $\wh\chi$ is the quotient projection. The contact form on $V$ is obtained by pushing the restriction of the standard contact form on $S^{8n-1}$ to $V$ by $\wh\chi$. The bundle projection is obtained as follows:
$$\pi \fc V = \Stie_2(\C^{2n})/\SU(2) \xrightarrow{S^1 = \Uni(2)/\SU(2)} \Stie_2(\C^{2n})/\Uni(2) = \Gr_2(\C^{2n}) = M\,.$$
To formulate our rigidity results, we need to construct some Legendrian\footnote{A submanifold of a contact manifold $(V,\xi)$ is Legendrian if it everywhere tangent to $\xi$ and has the maximal possible dimension $\frac 1 2 (\dim V -1)$.} submanifolds of $V$ and Lagrangian submanifolds of $M$. We define
$$L_1:=\wh\chi(\Stie_2(\R^{2n}))\,,$$
where $\Stie_2(\R^{2n})$ is the Stiefel variety of Euclidean orthonormal $2$-frames in $\R^{2n}$, naturally embedded in $\Stie_2(\C^{2n})$. Next, we identify $\C^{4n} = (\C^n)^4$ and define
$$L_2:=\wh\chi\big(\{(z,w;-\ol w,\ol z)\,|\,z,w\in\C^n\,,\|z\|^2+\|w\|^2=1)\}\big)\,.$$
This definition is related to item (iii) in Remark \ref{rem:after_main_result}, and the reader will notice that $(-\ol w,\ol z)$ is $j(z,w)$ if we identify $\C^{2n} = \HH^n$. We define
$$K_i:=\pi(L_i)\,,\quad Q_i:=\pi^{-1}(K_i)\,,\qquad i=1,2\,.$$

In the following theorem, $\pi \fc V\to M$ is the prequantization bundle we just described.
\begin{thm}\label{thm:rigidity_result}With the above notations,
\begin{enumerate}
  \item The submanifolds $L_i \subset V$ are Legendrian while $K_i\subset M$ are Lagrangian, for $i=1,2$;
  \item $L_1$ is diffeomorphic to the Grassmannian $\Gr_2^+(\R^{2n})$ of oriented real $2$-planes in $\R^{2n}$, $K_1$ to $\Gr_2(\R^{2n})$, the Grassmannian of nonoriented real $2$-planes, and $\pi|_{L_1}$ is the canonical double cover;
  \item $L_2$ is diffeomorphic to the quaternionic projective space $\HH P^{n-1}$, and $\pi|_{L_2}$ is a diffeomorphism onto $K_2$;
  \item For $i,j=1,2$, $K_i$ cannot be displaced from $K_j$ by a Hamiltonian isotopy;
  \item For $i,j=1,2$, $L_i$ cannot be displaced from $Q_j$ by a contact isotopy.
\end{enumerate}
\end{thm}
\noindent See Section \ref{ss:rigidity_results} for the proof. Our main contribution here is contained in item (v). The result in item (iv) was proved in \cite{Oh_FH_Lagr_intersections_hol_disks_I} for $i=j$ and in \cite{Iriyeh_Sakai_Tasaki_Lagr_HF_pair_real_forms_Hermition_symm_sp_cpt_type} for $i\neq j$, by computing the Lagrangian Floer homology $HF(K_i,K_j)$. It is included for completeness to illustrate the method of proof.

\subsection{Discussion}

Quasi-morphisms on $\wt\Cont_0(V,\xi)$ for contact manifolds $(V,\xi)$ have been constructed by
\begin{itemize}
  \item Poincar\'e with his rotation number on $\wt\Diff_0(S^1) \equiv \wt\Cont_0(S^1)$;
  \item Givental with his nonlinear Maslov index \cite{Givental_Nonlinear_gen_Maslov_index}, \cite{Ben_Simon_Nonlinear_Maslov_index_Calabi_homomorphism} for the standard contact $\R P^{2n-1}$;
  \item Borman--Zapolsky \cite{Borman_Zapolsky_Qms_Cont_rigidity} for prequantization bundles over even toric manifolds;
  \item Granja--Karshon--Pabiniak--Sandon \cite{Granja_Karshon_Pabiniak_Sandon_Maslov_index_lens_spaces} for lens spaces;
  \item Albers--Shelukhin--Zapolsky \cite{Albers_Shelukhin_Zapolsky} for lens spaces which are prequantization bundles over $\C P^n$.
\end{itemize}
The present paper is a natural continuation of the research presented in \cite{Borman_Zapolsky_Qms_Cont_rigidity}. We held hopes for a more general result, for instance including all complex Grassmannians, but in order to use the reduction technique for quasi-morphisms, Theorem \ref{thm:reduction_thm}, one must prove that the reduction locus is \emph{subheavy}, a notion defined below. This is a nontrivial task, and indeed a major hurdle one must overcome in order to employ the reduction theorem. The problem with Givental's nonlinear Maslov index is that one can prove subheaviness only for a few sets. The results of \cite{Albers_Shelukhin_Zapolsky} provide a quasi-morphism on $\wt\Cont_0$ of a lens space, which has potentially more subheavy sets. But unpublished calculations done together with Jack Smith show that even then not all complex Grassmannians will be included. The problem seems to be deep and it is far from being understood.

In general, quasi-morphisms on $\wt\Cont_0$ seem to be very hard to construct. In symplectic geometry, for example, there is now well-established technology for constructing quasi-morphisms on $\wt\Ham(M,\omega)$ where $(M,\omega)$ is a symplectic manifold, see \cite{Entov_Polterovich_Symp_QSs_semisimplicity_QH} and references therein. In contrast, all of the above constructions seem to be \emph{ad hoc}, and so far there is no unifying principle. Moreover, all of the above spaces have something to do with the standard symplectic vector space $\C^n$. It seems that for $\wt\Cont_0(V)$ the main problem is that, unlike in the symplectic case, there is no intrinsic notion of scale on a general contact manifold. All of the above examples belong to the class of prequantization bundles. This class does have an intrinsic notion of scale, given by the Reeb length of a fiber --- this is the so-called quantum scale. There are other manifestations of this scale, for example in contact (non)squeezing, see \cite{Eliashberg_Kim_Polterovich_Geom_cont_transfs_dom_orderability}. It would be very interesting to understand whether quasi-morphisms on general contactomorphism groups can be constructed, and whether there is some kind of general method of doing so, perhaps similar to \cite{Shelukhin_Action_hom_qms_mt_map_space_compatible_ACS}.

Rigidity results of Theorem \ref{thm:rigidity_result} belong to the context of rigid intersections of Legendrian (here $L_i$) and pre-Lagrangian (here $Q_i$) submanifolds. Previous results in this direction have been obtained in \cite{Eliashberg_Hofer_Salamon_Lagr_intersections_ct_geom}, and for prequantization bundles over non-aspherical monotone symplectic manifolds in \cite{Borman_Zapolsky_Qms_Cont_rigidity}. In analogy to Lagrangian Floer homology in the symplectic case, here there seems to be a unifying method, namely Lagrangian Floer homology and Legendrian contact homology in the symplectization of $V$, and the difficulties are then technical rather than conceptual.

\begin{acknow}The author was partially supported by the Israel Science Foundation grant 1825/14, and by grant 1281 from the GIF, the German--Israeli Foundation for Scientic Research and Development.
\end{acknow}

\section{Preliminaries}

The proofs of Theorems \ref{thm:main_result}, \ref{thm:rigidity_result} are based on the reduction technique developed in \cite{Borman_Zapolsky_Qms_Cont_rigidity}. To be able to use it, we need some definitions and basic results. All the material of this section is contained in Section 1 \emph{ibid}.

Let $(V,\xi)$ be a contact manifold and let $\alpha$ be a coorienting contact form for $\xi$. A submanifold $Y \subset V$ is $\alpha$-strictly coisotropic if $TY^{d\alpha} \subset TY$, where
$$T_yY^{d\alpha} = \{X\in T_yV\,|\,d\alpha(X,u)=0\text{ for all }u \in T_yY\}$$
for $y\in Y$. This is equivalent to $R_\alpha(y) \in T_yY$ for all $y \in Y$ and for the subbundle $TY\cap\xi|_Y$ being coisotropic in the symplectic vector bundle $(\xi|_Y,d\alpha)$. Now assume that $\mu \fc \wt\Cont_0(V,\xi)\to\R$ is a monotone quasi-morphism. A closed subset $Z \subset V$ is called $\mu$-subheavy if for all $h \in C^\infty(V)$ we have
$$h|_Z\equiv 0 \Rightarrow\mu(\wt\phi_h) = 0\,.$$
A closed subset $Z$ is called $\mu$-superheavy if for all $h \in C^\infty(V)$ we have
$$h|_Z>0 \Rightarrow\mu(\wt\phi_h) > 0\,.$$
This terminology was introduced in \cite{Borman_Zapolsky_Qms_Cont_rigidity} following \cite{Entov_Polterovich_Quasi_states_symplectic_intersections}. The property of being sub- or superheavy is independent of the contact form used to connect contact Hamiltonians and contact isotopies \cite[Proposition 1.10 (i)]{Borman_Zapolsky_Qms_Cont_rigidity}.

We will now describe a geometric setting which we refer to as contact reduction. Assume that in addition to $(V,\xi = \ker \alpha)$ we have another contact manifold $(\ol V,\ol\xi)$ with a chosen coorienting contact form $\ol\alpha$. We say that $(\ol V,\ol\xi)$ is obtained from $(V,\xi)$ by contact reduction if we have the diagram
\begin{equation}\label{eqn:geom_setting_reduction}
V \supset Y \xrightarrow{\rho}\ol V\,,
\end{equation}
where $Y$ is a closed connected $\alpha$-strictly coisotropic submanifold of $V$, and $\rho \fc Y \to \ol V$ is a fiber bundle such that $\rho^*\ol\alpha = \alpha|_Y$. We can now formulate the reduction theorem.
\begin{thm}[{\cite[Theorem 1.8]{Borman_Zapolsky_Qms_Cont_rigidity}}]\label{thm:reduction_thm}In the above geometric setting \eqref{eqn:geom_setting_reduction}, if $\mu \fc\wt\Cont_0(V,\xi) \to \R$ is a monotone quasi-morphism and $Y$ is $\mu$-subheavy, then there exists a unique monotone quasi-morphism $\ol\mu\fc\wt\Cont_0(\ol V,\ol\xi) \to \R$ satisfying
$$\ol\mu(\wt\phi_{\ol h}) = \mu(\wt\phi_h)$$
for all $\ol h \in C^\infty([0,1]\times \ol V)$ and $h \in C^\infty([0,1]\times V)$ such that $h|_Y = \rho^*\ol h$. If $\mu$ is continuous or has the vanishing property, then so does $\ol\mu$. \qed
\end{thm}


Next we list properties of sub- and superheavy sets:
\begin{prop}[{\cite[Proposition 1.10, Theorem 1.11]{Borman_Zapolsky_Qms_Cont_rigidity}}]\label{prop:properties_rigid_subsets_ct_mfds} Let $\mu \fc \wt\Cont_0(V,\xi) \to\R$ be a monotone quasi-morphism.
\begin{enumerate}
  \item If $Y,Z \subset V$ are closed subsets such that $Y \subset Z$ and $Y$ is $\mu$-subheavy or $\mu$-superheavy, then so is $Z$;
  \item A $\mu$-superheavy set is $\mu$-subheavy;
  \item The collections of $\mu$-subheavy and $\mu$-superheavy sets are invariant by the action of $\Cont_0(V,\xi)$;
  \item Every $\mu$-subheavy set intersects every $\mu$-superheavy set. \qed
\end{enumerate}
\end{prop}
\noindent One of useful corollaries of this is that a subheavy set cannot be displaced from a superheavy set by a contact isotopy, and in particular a superheavy set is nondisplaceable, see \cite[Corollary 1.12]{Borman_Zapolsky_Qms_Cont_rigidity}. This is the foundation of the rigidity results in Theorem \ref{thm:rigidity_result}.

The properties of being sub- or superheavy are preserved under contact reduction. More precisely, in the above setting we have:
\begin{thm}[{\cite[Theorem 1.14]{Borman_Zapolsky_Qms_Cont_rigidity}}]\label{thm:rigidity_descends_by_reduction} Let $\mu \fc \wt\Cont_0(V,\xi) \to \R$ be a monotone quasi-morphism and let $\ol\mu \fc \wt\Cont_0(\ol V,\ol\xi) \to \R$ be obtained from $\mu$ by reduction as in Theorem \ref{thm:reduction_thm}. If $Z \subset V$ is $\mu$-subheavy (superheavy) then $\rho(Z \cap Y)$ is $\ol\mu$-subheavy (superheavy). \qed
\end{thm}

Both for the construction of the quasi-morphism in Theorem \ref{thm:main_result} and for applications to symplectic and contact rigidity we need the connection between monotone quasi-morphisms on the total space of a prequantization bundle and symplectic quasi-states on its base. Let $(M,\omega)$ be a closed symplectic manifold. A functional $\zeta \fc C^\infty(M) \to \R$ is called a quasi-state if $\zeta(1) = 1$; $\zeta(F) \leq \zeta (G)$ whenever $F \leq G$; $\zeta$ is linear on Poisson-commuting subalgebras of $C^\infty(M)$. This notion appears in \cite{Entov_Polterovich_Quasi_states_symplectic_intersections} under the name ``symplectic quasi-state,'' to distinguish it from the more general notion of a topological quasi-state, due to Aarnes \cite{Aarnes_Qss_qms}. Since here we only deal with symplectic quasi-states, we drop the qualifier. In the presence of a quasi-state certain subsets exhibit rigidity. Following \cite{Entov_Polterovich_Quasi_states_symplectic_intersections}, a closed subset $X \subset M$ is called $\zeta$-superheavy if for all $H \in C^\infty(M)$
$$H|_X = c \in \R\Rightarrow \zeta(H) = c\,.$$
Any two superheavy subsets intersect, and if $\zeta$ is invariant under the action of the Hamiltonian group $\Ham(M,\omega)$ on $C^\infty(M)$, then the collection of superheavy subsets is likewise $\Ham(M,\omega)$-invariant, see \emph{ibid}.

Given a prequantization bundle $(V,\alpha,\pi,M,\omega)$, there is a natural short exact sequence of Lie algebras:
$$0 \to \R \to C^\infty(V)^{S^1} \to C^\infty(M)/\R \to 0\,,$$
where $C^\infty(V)^{S^1} \subset C^\infty(V)$ is the subalgebra of $S^1$-invariant functions. The Lie algebra structure on $C^\infty(M)/\R$ is given by the Poisson bracket, while on $C^\infty(V)$ it is the Lie bracket on the contact vector fields on $V$, transferred to $C^\infty(V)$ via $\alpha$.
This short exact sequence has a unique splitting by the Lie algebra morphism
$$\iota \fc C^\infty(M)/\R \to C^\infty(V)^{S^1}\,,\qquad H + \R \mapsto \pi^*H - \frac{\int_M H\,\omega^n}{\int_{M}\omega^n}\,,$$
where $n = \frac 1 2 \dim M$. There is then the induced homomorphism
$$\iota^* \fc \wt\Ham(M,\omega) \to \wt\Cont_0(V,\xi)\,.$$
If $\mu \fc \wt\Cont_0(V,\xi) \to \R$ is a quasi-morphism, then $\iota_*\mu := \mu \circ \iota^* \fc \wt\Ham(M,\omega) \to \R$ is also a quasi-morphism. Moreover if $\mu$ is monotone, then the functional $\zeta_\mu \fc C^\infty(M) \to \R$ defined by
\begin{equation}\label{eqn:qs_from_qm}
\zeta_\mu(H) = \frac{\mu(\wt\phi_{\pi^*H})}{\mu(\wt\phi_1)}
\end{equation}
is a quasi-state, and it can easily be seen to be $\Ham(M,\omega)$-invariant.

We will need the following connection between rigid subsets of $V$ and $M$:
\begin{prop}[{\cite[Proposition 1.20]{Borman_Zapolsky_Qms_Cont_rigidity}}]\label{prop:relation_superheavy_preq_bdle_total_sp_base} Let $\mu \fc \wt\Cont_0(V,\xi) \to \R$ be a monotone quasi-morphism and let $\zeta \fc C^\infty(M) \to \R$ be the quasi-state induced on $M$ from $\mu$. Then:
\begin{enumerate}
  \item If $Z \subset V$ is $\mu$-subheavy, then $\pi(Z) \subset M$ is $\zeta$-superheavy;
  \item If $X \subset M$ is $\zeta$-superheavy, then $\pi^{-1}(X) \subset V$ is $\mu$-superheavy. \qed
\end{enumerate}
\end{prop}

\section{Proofs}\label{s:proofs}

\subsection{Construction of the prequantization bundle}\label{ss:contruction_preq_bdle}

The prequantization bundle whose existence is asserted in Theorem \ref{thm:main_result} is obtained by symplectic and contact reduction. The reader is referred to \cite{Abraham_Marsden_Foundations_mechanics}, \cite{Geiges_Intro_ct_topology}, \cite{Borman_Zapolsky_Qms_Cont_rigidity}, \cite{McDuff_Salamon_Intro_sympl_topology} for background. We will use the following convention: if $(W,\eta)$, $(\ol W,\ol\eta)$ are symplectic manifolds and $\ol W$ is obtained from $W$ by symplectic reduction, then $\ol \eta$ is always the symplectic form coming from reduction, even if we do not explicitly state so, that is if $Z \subset W$ is the coisotropic reduction locus and $\pr \fc Z \to \ol W$ is the corresponding quotient map, then $\pr^*\ol\eta = \eta|_Z$.

The construction, though relatively elementary, is somewhat involved, so we divided it into a few subsections for clarity. We would also like to note that parts of this construction are already contained in \cite{Borman_Zapolsky_Qms_Cont_rigidity}, but we nevertheless present it in full detail for the sake of self-containment.

\subsubsection{Unitary reduction}

The base of the prequantization bundle in Theorem \ref{thm:main_result} is constructed by symplectic reduction of a complex vector space. The foundation for all the constructions below is the Hamiltonian action of the unitary group $\Uni(D)$ on $\C^D$. Let $\omega_{\std}$ be the standard symplectic form on $\C^D$. The action of $\Uni(D)$ on $\C^D$ has moment map $\Phi_{\Uni(D)} \fc \C^D \to \mfu(D)^*$ given by
$$\langle \Phi_{\Uni(D)}(z),X\rangle = \alpha_{\std,z}(Xz)\,,$$
where $z \in \C^D$, $X \in\mfu(D)$, $\alpha_{\std}$ is the Liouville form on $\C^D$ defined by $\alpha_{\std,z}=\frac 1 2\omega_{\std}(z,\cdot)$, and where we consider $X$ as a linear map on $\C^D$. In applications it is convenient to have another description of the moment map. To this end let us identify $\mfu(D)^* \simeq \mfu(D)$ by means of the inner product $(X,Y) \mapsto \tr(X^*Y)$ on $\mfu(D)$. The corresponding map $\wt \Phi_{\Uni(D)} \fc \C^D \to \mfu(D)$ writes
$$\wt \Phi_{\Uni(D)}(z) = \tfrac i 2\,zz^*\,,$$
where we consider $z$ as a column matrix. In symplectic reduction one considers coadjoint orbits. For our purposes we only consider coadjoint orbits which are fixed points. In general, a point $p \in \mfg^*$, where $\mfg$ is the Lie algebra of a compact connected Lie group $G$, is fixed by the coadjoint action if and only if $p$ vanishes on the commutator ideal $[\mfg,\mfg]$. The commutator ideal of $\mfu(D)$ is $\mfsu(D)$, which is the kernel of the functional $-\frac i 2 \tr \in \mfu(D)^*$, so all the fixed points of the coadjoint action are its multiples. The isomorphism $\mfu(D)^* \simeq \mfu(D)$ maps it to $\frac i 2 I \in \mfu(D)$.

\subsubsection{Toric manifolds}

First we describe how toric manifolds are constructed using symplectic reduction. We view $S^1$ as the unit circle in $\C^*$, and identify it with $\R/2\pi \Z$ by means of $\exp\fc \R/2\pi\Z\to S^1$, $t\mapsto e^{it}$. In particular we identify $\Lie(S^1) = \R$.

Let $\T^d = (S^1)^d$ be the maximal torus of $\Uni(d)$ consisting of diagonal matrices. The Lie algebra of $\T^d$ is $\R^d$. The embedding $\R^d \to \mfu(d)$ is given by $(t_1,\dots,t_d)\mapsto \diag (it_1,\dots,it_d)$. Since the action of $\Uni(d)$ on $\C^d$ is Hamiltonian, so is the restricted action of $\T^d$, and the corresponding moment map $\Phi_{\T^d} \fc \C^d \to \R^{d*}$ is the composition of $\Phi_{\Uni(d)}$ with the projection $\mfu(d)^*\to\R^{d*}$. Identifying $\R^{d*} \simeq \R^d$ in the obvious way, we get $\Phi_{\T^d}(z) = \frac 1 2(|z_1|^2,\dots,|z_d|^2)$. The point $-\frac i 2 \tr \in \mfu(d)^*$ maps to $\frac 1 2(1,\dots,1) \in \R^d \simeq \R^{d*}$ by the projection.

Let $N$ be a compact symplectic toric manifold, on which a torus $\T$ acts effectively via a moment map $\Phi_\T \fc N \to \mft^*$, where $\mft$ is the Lie algebra of $\T$. Let $\Delta \subset \mft^*$ be the moment polytope and let $\nu_1,\dots,\nu_d$ be the conormals to its faces. These are primitive vectors in the integer lattice $\mft_\Z = \ker (\exp \fc \mft \to \T)$. The moment polytope is then given by
$$\Delta = \{x \in \mft^*\,|\,\langle x,\nu_j\rangle+a_j \geq 0\text{ for }j=1,\dots, d\}\,,$$
where the $a_j$ are constants. Delzant \cite{Delzant_Hamiltoniens_periodiques_images_convexes_appl_moment} provided a way to exhibit $N$ as the symplectic reduction of $\C^d$ by the Hamiltonian action of a suitable subtorus of $\T^d$, as follows. Define the map
$$\beta \fc \R^d \to \mft\,,\qquad e_j \mapsto \nu_j\,,$$
where the $e_j$ are the standard basis vectors of $\R^d$. Let $\mfk:=\ker\beta$. Then $\mfk \subset \R^d$ is the Lie algebra of a subtorus $\K \subset \T^d$, which therefore has a Hamiltonian action on $\C^d$ given by the moment map $\Phi_\K \fc \C^d \to \mfk^*$, given by the composition of $\Phi_{\T^d}$ with the projection $\iota^*\fc\R^{d*} \to \mfk^*$, where $\iota \fc \mfk \to \R^d$ is the inclusion. Let $a = (a_1,\dots,a_d) \in \R^{d*}$ and define $c:=\iota^*a \in \mfk^*$. Then, up to a $\T$-equivariant symplectomorphism, $N$ is the reduction of $\C^d$ by the action of $\K$ at level $c$, that is $c$ is a regular value of $\Phi_\K$, $\K$ acts freely on the corresponding level set, and
$$N = \Phi_\K^{-1}(c)/\K\,,$$
together with the induced symplectic form. The quotient manifold inherits a Hamiltonian action of $\T^d/\K$, which is canonically isomorphic to $\T$, and the moment map is induced by the restriction of $\Phi_{\T^d}$ to $\Phi_\K^{-1}(c)$, up to shift. In what follows we will assume that our toric manifolds already come as symplectic reduction as discussed.
\begin{rem}\label{rem:condition_for_monotonicity_toric_mfd}
Note that $N$ is monotone if and only if $c$ is proportional to $\iota^*\big(\frac1 2(1,\dots,1)\big)$. This is because of the natural identifications $H^2(N;\R) = \mfk^*$, $H^2(N;\Z) = \mfk_\Z^*$, where $\mfk_\Z^* \subset \mfk^*$ is the integer lattice of functionals taking integer values on $\mfk_\Z \subset \mfk$. Under these identifications, the class of the symplectic form on $N$ is given by $c \in \mfk^*$ while $c_1(N)$ is $\iota^*(1,\dots,1)$.
\end{rem}

Assume now that $N$ is monotone, and, without loss of generality, that it is the symplectic reduction of $\C^d$ by the action of $\K \subset \T^d$ at the level $\iota^*\big(\frac 1 2 (1,\dots,1)\big)$. Since $N$ is compact, there exists a primitive weight vector $\gamma = (\gamma_1,\dots,\gamma_d) \in \N^d$ such that $\sum_{j=1}^{d}\gamma_j\nu_j = 0$. It follows that $\gamma \in \mfk$, and therefore the image of the embedding $S^1 \to \T^d$, $e^{it} \mapsto (e^{i\gamma_1t},\dots,e^{i\gamma_dt})$, is contained in $\K$. The moment map of the action of $S^1$ on $\C^d$ via this embedding is $\C^d\to\R^*$, $z \mapsto \frac 1 2 \sum_{j=1}^{d}\gamma_j|z_j|^2$, while the image of $c \in \mfk^*$ by the induced projection $\mfk^*\to\R^*\simeq\R$ equals $\frac 1 2\sum_{j=1}^{d}\gamma_j$. It follows that the level set $\Phi_\K^{-1}(c)$ lies in the sphere
$$S^{2d-1}_\gamma:=\big\{z\in\C^d\,|\,\textstyle \sum_{j=1}^{d}\gamma_j|z_j|^2 = \sum_{j=1}^{d}\gamma_j\big\}\,.$$
This will be important below.

\subsubsection{Complex Grassmannians and products}

Next we show how the complex Grassmannian $\Gr_k(\C^n)$ can be obtained by symplectic reduction from $\C^{nk}$. Viewing $\C^{nk}$ as the set of $n\times k$ matrices, the group $\Uni(k)$ acts on it by multiplication on the right. This action is unitary, and the corresponding embedding $\Uni(k) \to \Uni(nk)$ is given by mapping $A \in \Uni(k)$ to the block-diagonal $nk\times nk$ matrix whose only nonzero entries consist of $n$ diagonal $k\times k$ blocks all equal to $A$. The same procedure gives us the corresponding embedding of Lie algebras $\mfu(k) \to \mfu(nk)$. This action of $\Uni(k)$ on $\C^{nk}$ is Hamiltonian with moment map $\Phi_{\Uni(k)}\fc \C^{nk} \to \mfu(k)^*$, which is the composition of $\Phi_{\Uni(D)}$ with the projection $\mfu(nk)^* \to \mfu(k)^*$ dual to the above inclusion. Identifying $\mfu(k)^*\simeq\mfu(k)$ via the trace inner product gives us
$$\wt\Phi_{\Uni(k)} \fc \C^{nk} = (\C^n)^k \to \mfu(k)\,,\quad \wt\Phi_{\Uni(k)}(z_1,\dots,z_k) = \big(\tfrac i 2 (z_j,z_l)\big)_{jl}\,,$$
where for $z,w \in \C^n$ their Hermitian product is $(z,w) = \sum_{i=1}^{n}z_i\ol{w_i}$. The restriction of $-\frac i 2 \tr \in \mfu(nk)^*$ to $\mfu(k)$ is $-\frac {in} 2\tr \in \mfu(k)^*$. The corresponding level set of the moment map is $$\Phi_{\Uni(k)}^{-1}(-\tfrac{in}2\tr) = \{(z_1,\dots,z_k)\in(\C^n)^k\,|\,(z_j,z_l) = n\delta_{jl}\}\,,$$
that is the set of Hermitian orthogonal $k$-frames in $\C^n$ consisting of vectors having norm $\sqrt{n}$. It is in a natural bijection with the Stiefel variety $\Stie_k(\C^n)$ and the quotient is canonically $\Gr_k(\C^n)$ with the induced symplectic form. Note that the level set $\Phi_{\Uni(k)}^{-1}(-\frac{in}2\tr)$ is contained in the sphere $S^{2kn-1}_{\sqrt{kn}} \subset \C^{nk}$ where the subscript indicates the radius.

Next we show how to use symplectic reduction to obtain the monotone product
$$N \times \prod_{i=1}^{r}\Gr_{k_i}(\C^{n_i})\,,$$
where $N$ is a monotone compact toric manifold. Symplectic reduction behaves well with respect to products, that is, if we have a number of symplectic manifolds obtained by symplectic reduction, then so is their product. As we showed above, $N$ is the symplectic reduction of $\C^d$, where $d$ is the number of faces of its moment polytope, by a suitable torus $\K \subset \T^d$ at the level $c \in \mfk^*$. Also, $\Gr_{k_i}(\C^{n_i})$ is obtained by symplectic reduction of $\C^{n_ik_i}$ by the action of $\Uni(k_i)$ at the level $-\frac{in_i}2\tr \in \mfu(k_i)^*$. Therefore the sought-for product is the symplectic reduction of $\C^D$, where $D = d + \sum_{i=1}^{r}n_ik_i$, by the action of $\K \times \prod_{i=1}^{r}\Uni(k_i)$. Let us describe this in more detail. Let $G = \K \times \prod_{i=1}^{r}\Uni(k_i)$ and let $\mfg = \mfk \oplus \bigoplus_{i=1}^r\mfu(k_i)$ be its Lie algebra. Using the embeddings $\K \subset \Uni(d)$ and $\Uni(k_i) \to \Uni(n_ik_i)$ as above, we obtain the embedding
$$G \to \Uni(d) \times \prod_{i=1}^{r}\Uni(n_ik_i) \subset \Uni(D)\,,$$
where the last inclusion is the obvious block-diagonal embedding. The action of $G$ on $\C^D$ via this embedding is Hamiltonian with moment map
$$\Phi_G \fc \C^D = \C^d \times \prod_{i=1}^{r}\C^{n_ik_i} \to \mfg^*\,,\quad \Phi_G(z;w_1,\dots,w_r) = (\Phi_\K(z),\Phi_{\Uni(k_1)}(w_1),\dots,\Phi_{\Uni(k_r)}(w_r))\,,$$
where we identified $\mfg^*=\mfk^*\oplus\bigoplus_{i=1}^r\mfu(k_i)^*$. Let $p \in \mfg^*$ be the restriction of $-\frac i 2 \tr\in\mfu(D)^*$. Then $p = (c,-\frac{in_1}2\tr,\dots,-\frac{in_r}2\tr)$. It is a regular value of $\Phi_G$ and we have
$$\Phi_G^{-1}(p) = \Phi_\K^{-1}(c) \times \prod_{i=1}^{r}\Phi_{\Uni(k_i)}^{-1}(-\tfrac{in_i}2\tr) = \Phi_\K^{-1}(c) \times \prod_{i=1}^{r}\Stie_{k_i}(\C^{n_i})\,,$$
and the reduced manifold is
\begin{equation}\label{eqn:base_preq_bdle_as_reduction}
\Phi_G^{-1}(p)/G = N \times \prod_{i=1}^{r}\Gr_{k_i}(\C^{n_i})\,.
\end{equation}
We prove the following at the end of the next subsection:
\begin{lemma}\label{lemma:base_is_monotone}
This symplectic manifold is monotone.
\end{lemma}
The base of the prequantization bundle announced in Theorem \ref{thm:main_result} is the particular case of this construction when $N$ is even, $k_1=\dots=k_r = 2$ and $n_1,\dots,n_r$ are even numbers $\geq 4$.

\subsubsection{Reduction in stages}

It will be important for us to perform the above construction \emph{in stages}. Reduction in stages means that we first reduce by a normal subgroup of the given Lie group, and then reduce the resulting manifold by the quotient group. In our situation, we consider the torus
$$\T = \K \times (S^1)^r \subset \K \times \prod_{i=1}^{r}\Uni(k_i) = G\,,$$
where each circle is the center of the corresponding unitary group. It is the center of $G$; let $\mft$ be its Lie algebra. Let $\Phi_\T \fc \C^D \to \mft^*$ be the moment map of the Hamiltonian action of $\T$ on $\C^D$. The inclusion $\mft \to \mfu(D)$ is obtained as follows: we have the inclusions $\mfk \subset \R^d \subset \mfu(d)$, $\R \subset \mfu(n_ik_i)$, $t\mapsto itI$, and thus $\mft = \mfk \oplus \R^r \to \mfu(d)\oplus \bigoplus_{i=1}^r\mfu(n_ik_i) \subset \mfu(D)$. Therefore for the moment map we have
$$\Phi_\T \fc \C^D = \C^d\times\prod_{i=1}^{r}\C^{n_ik_i}\to\mfk^*\oplus\R^{r*}\,,\quad \Phi_\T(z;w_1,\dots,w_r) = (\Phi_\K(z),\tfrac 1 2 |w_1|^2,\dots,\tfrac 1 2 |w_r|^2)\,,$$
where we identified $\R^{r*} \simeq \R^r$. The element $-\frac i 2 \tr \in \mfu(D)^*$ restricts to
$$\wt c := \big(c,\tfrac{n_1k_1}2,\dots,\tfrac{n_rk_r}{2}\big) \in \mft^*\,,$$
which is a regular value of $\Phi_\T$. The corresponding level set is
$$\Phi_\T^{-1}(\wt c) = \Phi_\K^{-1}(c) \times \prod_{i=1}^{r}S^{2n_ik_i-1}_{\sqrt{k_in_i}}\,,$$
while the reduced manifold is
$$\wh M:=\Phi_\T^{-1}(\wt c)/\T = N \times \prod_{i=1}^{r}\C P^{n_ik_i-1}\,.$$
The quotient group
$$G/\T = \left(\K \times \prod_{i=1}^{r}\Uni(k_i)\right)\Big/\left(\K \times (S^1)^r\right) \simeq \prod_{i=1}^{r}\PU(k_i)$$
has Lie algebra $\mfg/\mft \simeq \bigoplus_{i=1}^r\mfpu(k_i)$. The induced action of $G/\T$ on $\Phi_\T^{-1}(\wt c)/\T$ is Hamiltonian with moment map
$$\Phi_{G/\T}\fc N \times \prod_{i=1}^{r}\C P^{n_ik_i-1} \to \bigoplus_{i=1}^r\mfpu(k_i)^*\,,$$
$$\Phi_{G/\T}(z;[w_1],\dots,[w_r])=(\Phi_{\Uni(k_1)}(w_1)+\tfrac{in_1}2\tr,\dots,\Phi_{\Uni(k_r)}(w_r)+\tfrac{in_r}2\tr)\,,$$
where $w_i \in S^{2n_ik_i-1}_{\sqrt{n_ik_i}}$ and $[w_i] \in \C P^{n_ik_i-1}$ is its image by the Hopf map. Note that $\mfpu(k_i)^*\subset \mfu(k_i)^*$ is the annihilator subspace of $\R \subset \mfu(k_i)$, and that $\Phi_{\Uni(k_i)}(w_i) + \frac{in_i}{2}\tr$ indeed vanishes on $\R$, meaning that the moment map is well-defined. The regular value at which we perform the reduction is $0\in(\mfg/\mft)^*$. The corresponding level set of the moment map is
$$\Phi_{G/\T}^{-1}(0) = N \times \prod_{i=1}^{r}\PV_{k_i}(\C^{n_i})\,,$$
where $\PV_{k}(\C^n) = \Stie_k(\C^n)/S^1$ is the projectivized Stiefel variety, and the reduced manifold is
\begin{equation}\label{eqn:base_reduced_by_PU}
\Phi_{G/\T}^{-1}(0)/(G/\T) = \left(N \times \prod_{i=1}^{r}\PV_{k_i}(\C^{n_i})\right)\Big/\left(\{1\}\times\prod_{i=1}^{r}\PU(k_i)\right) = N \times \prod_{i=1}^{r}\Gr_{k_i}(\C^{n_i})\,.
\end{equation}

\begin{proof}[Proof of Lemma \ref{lemma:base_is_monotone}]For a monotone symplectic manifold $(W,\omega^W)$ its monotonicity constant is the number $\lambda^W>0$ such that $c_1(W)=\lambda^W\omega^W$ as functionals on $H^2(W;\Z)$. Note that a product symplectic manifold is monotone if all the factors are monotone with the same monotonicity constant. It therefore suffices to show that all the factors in \eqref{eqn:base_preq_bdle_as_reduction} are monotone with the same monotonicity constant.

Let us first show that the toric manifold
$$\wh M=N \times \prod_{i=1}^{r}\C P^{k_in_i-1}$$
constructed above is monotone. Recall that it is obtained from $\C^D = \C^d \times \prod_{i=1}^{r}\C^{k_in_i}$ via symplectic reduction by the action of the torus $\T$. The moment map of its action is the composition $\kappa^* \circ \Phi_{\T^D}$ where $\Phi_{\T^D}\fc \C^D \to \R^{D*}$ is the moment map of the action of $\T^D$ while $\kappa^* \fc \R^{D*} \to \mft^*$ is the projection dual to the inclusion $\kappa \fc \mft \subset \R^D$. The reduction is at the level $\wt c \in \mft^*$ obtained by restricting $-\frac i 2\tr \in \mfu(D)^*$ to $\mft \subset \R^D \subset \mfu(D)$. The restriction of $-\frac i 2 \tr$ to $\R^D$ is the vector $\frac 1 2 (1,\dots,1)$, and $\wt c$ is its restriction to $\mft$. By Remark \ref{rem:condition_for_monotonicity_toric_mfd}, this means that the $\wh M$ is indeed monotone.

Let $\lambda > 0$ be its monotonicity constant, that is $c_1(\wh M) = \lambda\wh\omega$ as functionals on $H_2(\wh M;\Z)$, where $\wh \omega$ is the symplectic form coming from the standard symplectic form on $\C^D$ via reduction. It follows that the factors $N$ and $\C P^{k_in_i-1}$ are monotone with the same monotonicity constant $\lambda$. It therefore suffices to prove that $\Gr_{k_i}(\C^{n_i})$ has monotonicity constant $\lambda$. Note that $H_2(\Gr_{k_i}(\C^{n_i});\Z) = \Z$, generated by an appropriate Schubert cell, therefore it is automatically monotone. The only thing we need to take care of is the monotonicity constant.

To simplify notation, let $k=k_i$ and $n=n_i$. We have $H_2(\Gr_k(\C^n))=\pi_2(\Gr_2(\C^n))$ by the Hurewicz morphism, since $\Gr_k(\C^n)$ is simply connected. As a particular case of equation \eqref{eqn:base_reduced_by_PU}, $\Gr_k(\C^n)$ is obtained from $\C P^{kn-1}$ via symplectic reduction by the group $\PU(k)$. The level set of the moment map is $\PV_k(\C^n)$.

Let us show that the inclusion $\PV_k(\C^n) \subset \C P^{kn-1}$ induces an isomorphism on $\pi_2$. Indeed, consider the natural action of $\PGL(k) \supset \PU(k)$ on $\C P^{kn-1}$. The trace of $\PV_k(\C^n)$ by this action, that is $\PGL(k)\cdot \PV_k(\C^n)$, is the projectivization of the variety of $k$-frames in $\C^n$, therefore its complement $X \subset \C P^{kn-1}$ is the projectivization of the variety of singular $k$-frames in $\C^n$. The latter is an algebraic variety of complex codimension $n-k+1 \geq 3$ in $\C P^{kn-1}$ therefore real codimension at least $6$. It follows that a map $S^2 \to \C P^{kn-1}$ can be deformed away from $X$, and likewise for any homotopy of such maps, which proves the claim.

It follows that $\pi_2(\PV_k(\C^n)) = \Z$. The map on $\pi_2$ induced by the projection
$$\sigma \fc \PV_k(\C^n) \to \Gr_k(\C^n)$$
is the multiplication by $k$ viewed as a map $\Z \to \Z$. Let $\wt u \fc S^2 \to \PV_k(\C^n)$ be a smooth map representing the generator of $\pi_2(\PV_k(\C^n))$ having positive symplectic area. It means in particular that
$$\langle c_1(\C P^{kn-1}),[\wt u]\rangle = kn\,,$$
the minimal Chern number of $\C P^{kn-1}$. Let $u = \sigma\circ \wt u$. Denoting the symplectic form on $\Gr_k(\C^n)$ by $\omega$, we have $\sigma^*\omega=\omega^{\C P}|_{\PV_k(\C^n)}$, where $\omega^{\C P}$ is the symplectic form on $\C P^{kn-1}$. We then have
$$\langle \omega,[u]\rangle = \int_{S^2}u^*\omega = \int_{S^2}\wt u^*(\sigma^*\omega) = \int_{S^2}\wt u^*\omega^{\C P}=\langle \omega^{\C P},[\wt u]\rangle=\frac 1 \lambda \langle c_1(\C P^{kn-1}),[\wt u]\rangle = \frac{kn}\lambda\,.$$
On the other hand, $[u]$ is $k$ times the generator of $\Gr_k(\C^n)$ having positive area, which in particular means that
$$\langle c_1(\Gr_k(\C^n)),[u]\rangle = k\cdot n\,,$$
because the minimal Chern number of $\Gr_k(\C^n)$ is $n$. Thus we see that
$$\langle \omega,[u]\rangle = \frac{kn}{\lambda} = \frac 1 \lambda \langle c_1(\Gr_k(\C^n)),[u]\rangle\,,$$
which gives us the proportionality constant, that is $c_1(\Gr_k(\C^n)) = \lambda \omega$. This means that all the factors of $\wh M$ have the same monotonicity constant $\lambda$, therefore $\wh M$ is monotone.
\end{proof}

\subsubsection{Constructing the total space}\label{sss:constructing_total_space}

Next we describe the construction of the total space of the prequantization bundle. This is done using contact reduction. Recall the group $G = \K \times \prod_{i=1}^{r}\Uni(k_i) \subset \Uni(D)$ acting on $\C^D$ via the moment map $\Phi_G \fc \C^D \to \mfg^*$. We picked a weight vector $\gamma \in \N^d$ so that $\sum_{j=1}^{d}\gamma_j\nu_j = 0$. This implies that the embedding $S^1 \to \T^d$, $e^{it}\mapsto(e^{it\gamma_1},\dots,e^{it\gamma_d})$, has its image in $\K$. It follows that the embedding
$$S^1\to\T^D\,,\quad e^{it}\mapsto (e^{it\gamma_1},\dots,e^{it\gamma_d},e^{it},\dots,e^{it})$$
has its image in the central torus $\T \subset G$, because the first $d$ entries land in $\K$ while the rest land in the diagonal circles of the $\Uni(k_i)$. The action of $S^1$ on $\C^D$ via this embedding has moment map
$$\Phi_{S^1} \fc \C^D = \C^d \times \prod_{i=1}^{r}\C^{n_ik_i} \to \R^* \simeq \R\,,\quad\Phi_{S^1}(z;w_1,\dots,w_r) = \frac 1 2 \sum_{j=1}^{d}\gamma_j|z_j|^2+\frac 1 2 \sum_{i=1}^{r}|w_i|^2\,.$$
Let us define the weight vector $\Gamma = (\gamma_1,\dots,\gamma_d,1,\dots,1) \in \N^D$. Since $\Phi_{S^1}$ is the composition of $\Phi_G$ with the projection $\mfg^* \to \R^*$, the level set $\Phi_G^{-1}(p)$ is contained in the sphere
$$S^{2D-1}_{\Gamma} = \big\{(z;w_1,\dots,w_r)\,|\,\textstyle \sum_{j=1}^{d}\gamma_j|z_j|^2 + \sum_{i=1}^{r}|w_i|^2 = \sum_{j=1}^{d}\gamma_j+\sum_{i=1}^{r}n_ik_i\big\}\,,$$
since $\mfg^*\to\R^*$ maps $p\mapsto \frac 1 2\big(\sum_{j=1}^{d}\gamma_j + \sum_{i=1}^{r}n_ik_i\big)$. The subscript $\Gamma$ indicates that this is a sphere where the coordinates are weighted according to $\Gamma$, which is consistent with our definition of $S^{2d-1}_\gamma$ above. Let us define $\wh\alpha_\Gamma:=\alpha_{\std}|_{S^{2D-1}_\Gamma}$. This is a contact form on this sphere and thus $(S^{2D-1}_\Gamma,\ker\wh\alpha_\Gamma)$ is a contact manifold. Since $U(D)$ acts on $\C^D$ by maps preserving $\alpha$ and $G$ preserves $S_\Gamma^{2D-1}$, it follows that $G$ acts on it by contactomorphisms preserving $\wh\alpha_\Gamma$.

Consider the corresponding real projective space $\R P^{2D-1}_\Gamma=S^{2D-1}_\Gamma/\Z_2$ with the induced contact form $\alpha_\Gamma$ and contact structure $\xi_\Gamma:=\ker\alpha_\Gamma$. The total space of our prequantization bundle is obtained from $(\R P^{2D-1}_\Gamma,\xi_\Gamma)$ via contact reduction. The reason we have to use the real projective space as opposed to the sphere is that in Section \ref{ss:constructing_qm} we will use Givental's nonlinear Maslov index to construct a quasi-morphism using Theorem \ref{thm:reduction_thm}. Givental's index is a monotone quasi-morphism on $\wt\Cont_0(\R P^{2D-1})$, while no similar quasi-morphism exists for the sphere, since it is \emph{not orderable}, see \cite{Eliashberg_Kim_Polterovich_Geom_cont_transfs_dom_orderability} and \cite[Section 1.5]{Borman_Zapolsky_Qms_Cont_rigidity}. We will now describe this process in detail.

Recall $p \in \mfg^*$ and define $\mfg_0:=\ker p \subset \mfg$. Since $p$ is the restriction to $\mfg\subset \mfu(D)$ of $-\frac{i}{2}\tr \in \mfu(D)^*$, we have $\mfg_0 = \mfg\cap\ker(-\frac i 2\tr) = \mfg \cap \mfsu(D)$. It follows that $\mfg_0$ is the Lie algebra of the Lie group $G \cap \SU(D)$, and we let $G_0$ be its identity component. For $X \in \mfg_0$ and $z \in \Phi_G^{-1}(p)$ we have
$$\alpha_{\std,z}(Xz) = \langle\Phi_G(z),X\rangle = \langle p, X \rangle = 0\,,$$
meaning the infinitesimal action of $G_0$ on $\Phi_G^{-1}(p)$ is tangent to $\ker\wh\alpha_\Gamma|_{\Phi^{-1}(p)}$. It follows that the orbits of $G_0$ are tangent to $\ker\wh\alpha_\Gamma|_{\Phi^{-1}(p)}$, and since $G_0 \subset G$ preserves $\wh\alpha_\Gamma$, it descends to a unique $1$-form on the quotient $\Phi_G^{-1}(p)/G_0$.

To get the total space of the prequantization bundle of Theorem \ref{thm:main_result}, we need the corresponding reduction of $\R P^{2D-1}_\Gamma$. At this point we assume that $N$ is even, which implies that $-I \in \Uni(d)$ lies in the torus $\K$ (see \cite[Lemma 2.1]{Borman_Zapolsky_Qms_Cont_rigidity}), which implies that $-I \in \Uni(D)$ lies in $G$. Thus $\Z_2 = \{I,-I\} \subset \Uni(D)$ is contained in $G$. We let $G_0^\tau$ be the subgroup of $G$ generated by $G_0$ and $\Z_2$. Since $\Z_2$ is central in $G$, $G_0^\tau = G_0$ if $\Z_2 \subset G_0$ and $G_0^\tau = G_0\cdot \Z_2$ otherwise. Consider the subset
$$\Phi_G^{-1}(p)/\Z_2 \subset \R P^{2D-1}_\Gamma\,.$$
The group $G_0^\tau/\Z_2$ acts on it freely by $\alpha_\Gamma$-preserving diffeomorphisms, and the orbits of the action are tangent to $\ker\alpha_\Gamma$, therefore we have a well-defined $1$-form $\alpha$ on the quotient
$$V:=\frac{\Phi_G^{-1}/\Z_2}{G_0^\tau/\Z_2} = \Phi_G^{-1}(p)/G_0^\tau$$
satisfying
$$\chi^*\alpha = \alpha_\Gamma|_{\Phi_G^{-1}/\Z_2}\,,$$
where $\chi \fc \Phi_G^{-1}(p)/\Z_2 \to V$ is the quotient projection. This is the total space of the prequantization bundle announced in Theorem \ref{thm:main_result}. The projection to $M = \Phi_G^{-1}(p)/G$ is the obvious principal $G/G_0^\tau$-bundle
$$\pi\fc V = \Phi_G^{-1}(p)/G_0^\tau \to \Phi_G^{-1}(p)/G = M\,.$$
The Lie group $G/G_0^\tau$ is compact, connected, and $1$-dimensional, therefore it is a circle. Let $\omega$ denote the symplectic form on $M$ coming from the above reduction procedure, that is it is defined by $\sigma^*\omega = \omega_{\std}|_{\Phi_G^{-1}(p)}$, where $\sigma \fc \Phi_G^{-1}(p) \to M$ is the quotient projection. We then have
\begin{prop}\label{prop:what_constructed_is_preq_bundle} $(V,\alpha,\pi,M,\omega)$ is a prequantization bundle.
\end{prop}
\begin{proof}Let us show that $\pi^*\omega = d\alpha$. Both $V$ and $M$ are quotients of the subset
$$\Phi_G^{-1}(p) \subset S^{2D-1}_\Gamma\,.$$
Let $\wh\chi \fc \Phi_G^{-1}(p) \to V$ be the quotient projection and recall the projection $\sigma \fc \Phi_G^{-1}(p) \to M$. Then from the definitions we have
$$\sigma^*\omega = \omega_{\std}|_{\Phi_G^{-1}(p)}\quad\text{and}\quad \wh\chi^*\alpha = \wh\alpha_\Gamma|_{\Phi_G^{-1}(p)}\,.$$
Since $\wh\alpha_\Gamma= \alpha_{\std}|_{S^{2D-1}_\Gamma}$, we have $\wh\chi^*\alpha = \alpha_{\std}|_{\Phi_G^{-1}(p)}$. Since $\sigma = \pi\circ \wh\chi$ and $\omega_{\std} = d\alpha_{\std}$, this means that
$$\wh\chi^*(d\alpha - \pi^*\omega) = d(\wh\chi^*\alpha) - \sigma^*\omega = d(\alpha_{\std}|_{\Phi_G^{-1}(p)}) - \omega_{\std}|_{\Phi_G^{-1}(p)} = 0\,.$$
It follows that $\wh\chi^*(d\alpha-\pi^*\omega) = 0$. Since $\wh\chi$ is a submersion, its differential is surjective everywhere and therefore $\wh\chi^*$ is injective on forms, meaning that $d\alpha = \pi^*\omega$. Since $G$ preserves $\alpha_\Gamma$, $G/G_0^\tau$ preserves $\alpha$, which means that $\alpha$ is an $S^1=G/G_0^\tau$-invariant $1$-form. The Reeb vector field of $\wh\alpha_\Gamma$ on $S^{2D-1}_\Gamma$ is tangent to $\Phi_G^{-1}(p)$, in fact it comes from the infinitesimal action of $\mfg$, therefore it projects to the vector field on $V$ generating the infinitesimal $S^1$-action, which in turn implies that $\alpha$ is nonzero when evaluated on this vector field. This means that $\alpha$ is a connection form on $V$ satisfying $d\alpha = \pi^*\omega$, meaning that $(V,\alpha,\pi,M,\omega)$ is indeed a prequantization bundle.
\end{proof}
\noindent This finishes the construction of the prequantization bundle whose existence is asserted in Theorem \ref{thm:main_result}.

Similarly to what we did for the base $M$, we will need to obtain the total space $V$ via reduction in stages. Recall the central torus $\T \subset G$ and the regular value $\wt c \in \mft^*$ of the moment map $\Phi_\T$. Let $\mft_0 = \ker \wt c \subset \mft$ and let $\T_0 \subset \T$ be the subtorus with Lie algebra $\mft_0$. Since $\mft_0 = \mft\cap\mfsu(D)$, $\T_0$ is well-defined as the identity component of the abelian Lie group $\T \cap \SU(D)$ whose Lie algebra is $\mft\cap\mfsu(D)$. Since $N$ is assumed even, $\Z_2 = \{I,-I\} \subset \Uni(D)$ is contained in $\T$. We let $\T_0^\tau = \T_0\cdot\Z_2 \subset \T$. The group $\T_0^\tau$ acts freely on the manifold
$$\Phi_\T^{-1}(\wt c)/\Z_2 \subset \R P^{2D-1}_\Gamma\,,$$
and the orbits of its action are tangent to $\ker \alpha_\Gamma$. Since $\T$ preserves $\alpha_\Gamma$, it descends to a well-defined $1$-form $\wh\alpha$ on the quotient
$$\wh V:=\frac{\Phi_\T^{-1}(\wt c)/\Z_2}{\T_0^\tau/\Z_2} = \Phi_\T^{-1}(\wt c)/\T_0^\tau\,,$$
which forms a principal $\T/\T_0^\tau$-bundle over the symplectic toric manifold
$$\wh M = \Phi_\T^{-1}(\wt c)/\T = N \times \prod_{i=1}^{r}\C P^{n_ik_i-1}\,.$$
Similarly to Proposition \ref{prop:what_constructed_is_preq_bundle}, we can show that this is a prequantization bundle, and in fact it is proved in \cite{Borman_Zapolsky_Qms_Cont_rigidity}. Next, the group $G/\T_0^\tau$ acts on $\wh V=\Phi_\T^{-1}(\wt c)/\T_0^\tau$ by $\wh\alpha$-preserving contactomorphisms. Its action on the subset
$$\Phi_G^{-1}(p)/\T_0^\tau \subset \wh V$$
is free and the restricted action of $G_0^\tau/\T_0^\tau \subset G/\T_0^\tau$ has orbits tangent to the kernel of $\wh\alpha$, which means that it descends to a well-defined $1$-form on the quotient
$$V = \frac{\Phi_G^{-1}(p)/\T_0^\tau}{G_0^\tau/\T_0^\tau} = \Phi_G^{-1}(p)/G_0^\tau\,.$$
This is exactly the contact form $\alpha$ constructed above. The following diagrams summarize the above constructions in one place. For the base of the prequantization bundle we have
\begin{multline*}
\C^D \hookleftarrow \Phi_\T^{-1}(\wt c) \rightarrow \wh M = \Phi_\T^{-1}(\wt c)/\T = N \times \prod_{i=1}^{r}\C P^{n_ik_i-1} \hookleftarrow \Phi_G^{-1}(p)/\T = N \times\prod_{i=1}^{r}\PV_{k_i}(\C^{n_i})\rightarrow \\ \rightarrow M =\Phi_G^{-1}(p)/G = N \times \prod_{i=1}^{r}\Gr_{k_i}(\C^{n_i})\,,
\end{multline*}
while for the total space we have
$$\R P^{2D-1}_\Gamma\hookleftarrow \Phi_\T^{-1}(\wt c)/\Z_2 \to \wh V = \Phi_\T^{-1}(\wt c)/\T_0^\tau \hookleftarrow \Phi_G^{-1}(p)/\T_0^\tau \to V =\Phi_G^{-1}(p)/G_0^\tau\,.$$

For the construction of the quasi-morphism in Theorem \ref{thm:main_result} we will need the following.
\begin{lemma}\label{lemma:strictly_coiso}
The subset $\Phi_G^{-1}(p)/\T_0^\tau \subset \wh V$ is $\wh\alpha$-strictly coisotropic.
\end{lemma}
\begin{proof}
One can adapt the argument in \cite[Section 2.2]{Borman_Zapolsky_Qms_Cont_rigidity} to this case. We will present a general argument. Certains parts of it can be found in \cite[Section 7.7]{Geiges_Intro_ct_topology}, where we refer the reader for general background on contact reduction.

Let a compact connected Lie group $H$ with Lie algebra $\mfh$ act on a contact manifold $(P,\eta=\ker\beta)$ with a fixed choice of coorienting contact form, and assume that the action preserves $\beta$. In this case the action is expressible in terms of the associated moment map
$$\Phi_H \fc P \to \mfh^*\,,\quad \langle\Phi_H(x),X\rangle = \beta_x(X(x))\,,$$
where $x \in P$ and where by abuse of notation $X \in \mfh$ also denotes the vector field obtained by the infinitesimal action of $X$ on $P$. We claim that if $0 \in \mfh^*$ is a regular value of $\Phi_H$, then $Y=\Phi_H^{-1}(0)$ is strictly $\beta$-coisotropic. We wish to prove that
$$T_xY^{d\beta} \subset T_xY \quad \text{for all }x \in Y\,.$$
Let $x \in Y$ and consider the action map $a \fc H \to P$, $h \mapsto h\cdot x$. For $X \in \mfh$ we have
$$(a^*\beta)_1(X) = \beta_x(X(x)) = \langle \Phi_H(x),X\rangle = 0\,.$$
Since $\beta$ is $H$-invariant $a$ is $H$-equivariant, $a^*\beta$ is left-invariant, therefore $a^*\beta = 0$. In particular $a^*d\beta = 0$, which means that $\mfh(x):=d_0a(\mfh)$ is an isotropic subspace of $(\eta_x,d_x\beta)$.

We have $T_xY = \ker d_x\Phi_H$. The differential $d_x\Phi_H \fc T_xP \to \mfh^*$ is computed as follows:
$$\langle d_x\Phi_H(u),X\rangle = d_x(\iota_X\beta)(u)=(\underbrace{\cL_X\beta}_{=0} - \iota_X d\beta)_x(u) = d\beta(u,X(x))\,,$$
therefore
$$T_xY = \ker d_x\Phi_H = \{u\in T_xP\,|\, d_x\beta(u,X(x)) = 0\text{ for all }X \in \mfh\} = \mfh(x)^{d\beta}\,.$$
Since $\mfh(x) \subset \eta_x$, $d_x\beta$ is a symplectic form on $\eta_x$, and $\iota_{R_\beta}d\beta = 0$, $T_xY = \mfh(x)^{d\beta}$ is the direct sum of $\R R_\beta(x)$ and the $d_x\beta$-complement of $\mfh(x)$ inside $\eta_x$. It follows that $T_xY^{d\beta}$ is the direct sum of $\R R_\beta(x)$ and the complement of the complement of $\mfh(x)$ inside $\eta_x$, which equals $\mfh(x)$, that is
$$T_xY^{d\beta} = \R R_\beta(x) \oplus\mfh(x)\,.$$
Since $\mfh(x) \subset \eta_x$ is isotropic, it is contained in its own complement, which implies
$$T_xY^{d\beta} \subset \R R_\beta(x) \oplus \mfh(x)^{d\beta}\text{(inside }\eta_x\text) = T_xY\,,$$
as claimed.

In our situation $P$ = $\wh V = \Phi_\T^{-1}(\wt c)/\T_0^\tau$, $\beta = \wh \alpha$, and $H = G_0^\tau/\T_0^\tau = G_0/\T_0$. The moment map
$$\Phi_H \fc \wh V \to \mfh^* = (\mfg_0/\mft_0)^*$$
is induced by the composition of the restriction
$$\Phi_G|_{\Phi_\T^{-1}(\wt c)} \fc \Phi_\T^{-1}(\wt c) \to p + (\mfg/\R)^*$$
with the projection
$$p + (\mfg/\R)^* \to \mfg_0^*\,.$$
We have
$$\Phi_G^{-1}(p)/\T_0^\tau = \Phi_H^{-1}(0)\,,$$
which is $\wh\alpha$-strictly coisotropic as we have just shown.
\end{proof}

\subsection{Constructing the quasi-morphism}\label{ss:constructing_qm}

Here we will use Givental's asymptotic nonlinear Maslov index constructed in \cite{Givental_Nonlinear_gen_Maslov_index}. It is a quasi-morphism \cite{Ben_Simon_Nonlinear_Maslov_index_Calabi_homomorphism}:
$$\mu_{\text{Giv}} \fc \wt\Cont_0(\R P^{2D-1}_\Gamma,\xi_\Gamma) \to \R\,.$$
We refer the reader to \cite{Borman_Zapolsky_Qms_Cont_rigidity} for more background on this point.

In Section \ref{ss:contruction_preq_bdle} above we constructed the prequantization bundle
$$(\wh V,\wh\alpha,\wh\pi,\wh M,\wh \omega)$$
where
$$\wh M=\Phi_\T^{-1}(\wt c)/\T = N \times \prod_{i=1}^{r}\C P^{k_in_i-1}\,,$$
$\wh\omega$ is the symplectic form induced on it from $\C^D$ via symplectic reduction,
$$\wh V = \Phi_\T^{-1}(\wt c)/\T_0^\tau\,,$$
and $\wh\alpha$ is the contact form obtained by pushing $\alpha_\Gamma|_{\Phi_\T^{-1}(\wt c)/\Z_2}$ down to $\wh V$. This prequantization bundle is the one constructed in \cite{Borman_Zapolsky_Qms_Cont_rigidity} where we note that $\wh M$ is an even toric symplectic manifold. The main result of \cite{Borman_Zapolsky_Qms_Cont_rigidity} is an application of Theorem \ref{thm:reduction_thm}, where $\wh V$ is obtained from the real projective space $\R P^{2D-1}_\Gamma$ by contact reduction. Theorem \ref{thm:reduction_thm} then yields a quasi-morphism
$$\wh \mu \fc \wt\Cont_0(\wh V,\wh\xi=\ker\wh\alpha)\to\R\,,$$
by reducing $\mu_{\text{Giv}}$. It is monotone, continuous, and satisfies the vanishing property. The quasi-morphism announced in Theorem \ref{thm:main_result} is obtained from $\wh\mu$ by an additional application of the reduction theorem. Consider the diagram
$$\xymatrix{\wh V = \Phi_\T^{-1}(\wt c)/\T_0^\tau \ar[d]^{\wh \pi} & Y = \Phi_G^{-1}(p)/\T_0^\tau \ar[l]_{\supset} \ar[d]^{\wh \pi|_Y} \ar[r]^{\rho} & V = \Phi_G^{-1}(p)/G_0^\tau \ar[d]^{\pi} \\
\wh M = \Phi_\T^{-1}(\wt c)/\T & X = \Phi_G^{-1}(p)/\T \ar[l]_{\supset} \ar[r]& M = \Phi_G^{-1}(p)/G}$$
The top line is an instance of the geometric setting \eqref{eqn:geom_setting_reduction} in which the reduction theorem applies: $(\wh V,\wh \alpha)$ and $(V,\alpha)$ are closed connected contact manifolds with fixed choices of contact forms, $Y \subset \wh V$ is an $\wh\alpha$-strictly coisotropic submanifold (Lemma \ref{lemma:strictly_coiso}), and $\rho \fc Y \to V$ is a fiber bundle such that $\rho^*\alpha = \wh\alpha|_Y$. The missing crucial ingredient is to show that $Y$ is $\wh\mu$-subheavy. It is at this point that we must use our remaining assumption, namely that the base of the prequantization bundle has the special form
$$M = N \times \prod_{i=1}^{r}\Gr_2(\C^{2n_i})\,,$$
that is we only have \emph{Grassmannians of $2$-planes in even-dimensional complex spaces}. Let us assume this until the end of the proof. For this special case we obtain
$$\wh M = N \times \prod_{i=1}^{r}\C P^{4n_i-1}\,.$$

Let $\wh\zeta$ be the quasi-state on $\wh M$ induced from $\wh\mu$ by the formula \eqref{eqn:qs_from_qm}. Since $Y = \wh\pi^{-1}(X)$ and thanks to the relationship between $\wh\zeta$-superheaviness in $\wh M$ and $\wh\mu$-superheaviness in $\wh V$, see Proposition \ref{prop:relation_superheavy_preq_bdle_total_sp_base}, in order to show that $Y$ is $\wh\mu$-subheavy, it suffices to show that $X$ is $\wh\zeta$-superheavy. It is enough to show in fact that it contains a $\wh\zeta$-superheavy subset, see Proposition \ref{prop:properties_rigid_subsets_ct_mfds}, which is what we are going to do.

Theorem \ref{thm:rigidity_descends_by_reduction} says that under reduction sub- and superheaviness are preserved. As we described above, $\wh V$ is obtained from $\R P^{2D-1}_\Gamma$ by contact reduction, using the diagram
$$\R P^{2D-1}_\Gamma\supset \Phi_\T^{-1}(\wt c)/\Z_2 \xrightarrow{\theta} \wh V = \Phi_\T^{-1}(\wt c)/\T_0^\tau\,.$$
Since the quasi-morphism $\wh\mu$ is the reduction of $\mu_{\text{Giv}}$ on $\R P^{2D-1}_\Gamma$, $\mu_{\text{Giv}}$-sub- and superheavy sets reduce to $\wh\mu$-sub- and superheavy sets in $V$, respectively. In \cite{Borman_Zapolsky_Qms_Cont_rigidity} it was shown that the real part $\R P^{D-1} = S^{D-1}/\Z_2 \subset \R P^{2D-1}_\Gamma$, where $S^{D-1} = S^{2D-1}_\Gamma\cap \R^D \oplus 0 \subset \C^D$, which is a Legendrian submanifold, is $\mu_{\text{Giv}}$-subheavy. Therefore its reduction
$$L:=\theta(\R P^{D-1} \cap \Phi_\T^{-1}(\wt c)/\Z_2) \subset \wh V$$
is $\wh\mu$-subheavy. Thanks to Proposition \ref{prop:relation_superheavy_preq_bdle_total_sp_base}, the projection $K:=\wh\pi(L) \subset \wh M$ is $\wh \zeta$-superheavy. In \cite{Borman_Zapolsky_Qms_Cont_rigidity} it was shown that $K$ is the real part of the toric manifold $\wh M$, that is
$$K = N_\R \times \prod_{i=1}^{r}\R P^{4n_i-1}\,,$$
where $N_\R \subset N$ is the real part of $N$. The real part of a toric manifold is the fixed point set of the unique antisymplectic involution which anticommutes with the torus action. Since the collection of $\wh\zeta$-superheavy sets is invariant by the Hamiltonian group of $\wh M$, it now suffices to find a Hamiltonian diffeomorphism mapping $K$ into $X$. Recall that
$$X = N \times \prod_{i=1}^{r}\PV_2(\C^{2n_i})\,.$$
It is therefore enough to find a Hamiltonian diffeomorphism of $\C P^{4n-1}$ mapping its real part $\R P^{4n-1}$ into $\PV_2(\C^{2n})$. This is the part where our assumption that we are dealing with Grassmannians of $2$-planes in even-dimensional spaces comes to light.

Identify the quaternions $\HH$ with $\C^2$ by means of $\C^2 \to \HH$, $(z,w) \mapsto z+wj$. Also identify $\HH = \R^4$ by means of coordinates. Now consider the map
\begin{equation}\label{eqn:the_map_b}
b\fc \R^4 = \HH \to \C^2 \times \C^2= \HH \times \HH\,,\quad q \mapsto \tfrac 1 {\sqrt 2}(q,jq)\,.
\end{equation}

Writing $q = z+wj$, this translates into
$$(z,w)\mapsto \tfrac 1 {\sqrt 2}(z,w;-\ol w,\ol z)\,.$$
The main features of this map are (i) it transforms the Euclidean inner product on $\R^4$ into the Hermitian inner product on $\C^4$, namely
$$(b(q),b(q')) = \langle q, q'\rangle\,,$$
and (ii) its image is contained in the subvariety of Hermitian orthogonal pairs of vectors in $\C^2$. Taking a product of $n$ copies of this map, we obtain
$$\R^{4n} = \HH^n \to \C^{2n} \times \C^{2n} = \HH^n \times \HH^n\,,\quad (q_1,\dots,q_n)\mapsto \tfrac 1 {\sqrt 2}(q_1,\dots,q_n;jq_1,\dots,jq_n)\,,$$
which shares the same properties, namely it preserves the inner products and its image is contained in the set of Hermitian orthogonal pairs of vectors in $\C^{2n}$. Extending by $\C$-linearity, we obtain a map
\begin{equation}\label{eqn:unitary_extension_of_b}
b_\C \fc \C^{4n} = \R^{4n} \otimes_\R\C \to \C^{2n} \times \C^{2n} = \C^{4n}\,,
\end{equation}
which explicitly can be written as follows:
$$b_\C \fc \HH^n \times \HH^n \to \HH^n \times \HH^n\,, \quad b_\C(q,q') = \tfrac 1 {\sqrt 2}(q+iq',jq+kq')\,,$$
where in the first product $\HH^n \times \HH^n$ the first factor is identified with $\R^{4n}$ while the second factor is identified with $i\R^{4n}$ in the decomposition $\R^{4n} \otimes \C = \R^{4n} \oplus i\R^{4n}$. Property (i) above means that this extension is unitary. Moreover, it maps the real unit sphere $S^{4n-1} \subset \R^{4n} \oplus 0 \subset \C^{4n}$ into the Stiefel variety $\Stie_2(\C^{2n}) \subset \C^{2n} \times \C^{2n}$. Let us denote the map induced by $b_\C$ on $\C P^{4n-1}$ by $B \in \PU(4n)$. It follows that
$$B(\R P^{4n-1}) \subset \PV_2(\C^{2n})\,.$$
Since the projective unitary group is contained in the Hamiltonian group, we found a Hamiltonian diffeomorphism mapping the Lagrangian $\R P^{4n-1}$ into the projective Stiefel variety $\PV_2(\C^{2n})$. Taking a product of a number of copies of this map, we have therefore found a Hamiltonian diffeomorphism of $\wh M$ mapping $K$ into $X$, proving that $X$ is $\wh\zeta$-superheavy, and consequently that $Y$ is $\wh\mu$-superheavy, in particular $\wh\mu$-subheavy (Proposition \ref{prop:properties_rigid_subsets_ct_mfds}), which allows us to invoke the reduction theorem (Theorem \ref{thm:reduction_thm}) to obtain a quasi-morphism
$$\mu\fc\wt\Cont_0(V,\xi)\to\R\,.$$
The reduction theorem also shows that $\mu$ is monotone, continuous, and has the vanishing property. This completes the proof of Theorem \ref{thm:main_result}.

\subsection{Rigidity results}\label{ss:rigidity_results}

Here we prove Theorem \ref{thm:rigidity_result}. The main objective is to construct the Legendrians $L_1,L_2$ and to prove that they are subheavy. Recall that Theorem \ref{thm:reduction_thm} allows us to induce quasi-morphisms through the procedure of contact reduction. Let us describe how the prequantization bundle $\pi \fc V \to M$ appearing in Theorem \ref{thm:rigidity_result} is constructed. Keeping the notations of Section \ref{ss:contruction_preq_bdle}, this is the special case when $N = \pt$, $r=1$. The base $\Gr_2(\C^{2n})$ of the prequantization bundle is constructed as follows. The group $\Uni(2)$ acts on $\C^{4n}=(\C^{2n})^2$ via the moment map $\Phi_{\Uni(2)}\fc (\C^{2n})^2 \to \mfu(2)^*$. The value $p = -in\tr \in \mfu(2)^*$ is regular, and the corresponding level set of the moment map is
$$\Phi_{\Uni(2)}^{-1}(p) = \{(z,w)\in(\C^{2n})^2)\,|\,\|z\|^2 = \|w\|^2 = 2n\,,(z,w) = 0\} \simeq \Stie_2(\C^{2n})\,,$$
where the identification sends $(z,w)\mapsto (\frac{z}{\sqrt{2n}},\frac{w}{\sqrt{2n}})$. We will identify the two manifolds by this map in what follows. The base $M$ is then the quotient
$$M= \Phi_{\Uni(2)}^{-1}(p)/\Uni(2) = \Stie_2(\C^{2n})/\Uni(2) = \Gr_2(\C^{2n})\,.$$

As for total space, first we need to identify the group $G_0 \subset G = \Uni(2)$. Since $\ker p = \mfsu(2) \subset \mfu(2)$, it follows that $G_0 = \SU(2)$, and since $-I \in \SU(2)$, we have $G_0^\tau = G_0$. Thus the total space $V$ of the prequantization bundle is
$$V = \frac{\Stie_2(\C^{2n})/\Z_2}{\SU(2)/\Z_2} = \Stie_2(\C^{2n})/\SU(2)\,,$$
and the projection $\pi \fc V \to M$ is
$$V = \Stie_2(\C^{2n})/\SU(2) \xrightarrow{\Uni(2)/\SU(2)=S^1} \Stie_2(\C^{2n})/\Uni(2) = M\,,$$
The contact manifold $(V,\xi = \ker \alpha)$ is obtained from $\R P^{8n-1}$ by contact reduction, as follows. The level set $\Phi_{\Uni(2)}^{-1}(p) = \Stie_2(\C^{2n})$ is contained in the sphere $S^{8n-1} \subset \C^{4n}$. Note that this is the standard round sphere (up to a change of radius), since all the components of the weight vector $\Gamma$ are $1$. Quotienting out $\Z_2$, we get the following diagram:
\begin{equation}\label{eqn:diagram_reduction_RP_to_Grassmannian}
\R P^{8n-1} \supset \Stie_2(\C^{2n})/\Z_2 \xrightarrow{\chi}\Stie_2(\C^{2n})/\SU(2) = V\,.
\end{equation}
Note that for the reduction in stages described in Section \ref{sss:constructing_total_space}, we have the following: $\T \subset \Uni(2)$ is the central circle, $\T_0$ is the trivial subgroup, $\T_0^\tau = \Z_2$. It follows that\
$$\Phi_\T^{-1}(\wt c)/\Z_2 = \wh V = \Phi_\T^{-1}/\T_0^\tau = \R P^{8n-1}\,,$$
and therefore the first stage is trivial and we get $\wh\mu = \mu_{\text{Giv}}$.

The quasi-morphism $\mu \fc \wt\Cont_0(V,\xi) \to \R$ is the reduction of $\mu_{\text{Giv}} \fc \wt\Cont_0(\R P^{8n-1}) \to \R$ by means of diagram \eqref{eqn:diagram_reduction_RP_to_Grassmannian}, as described in Theorem \ref{thm:reduction_thm}. Theorem \ref{thm:rigidity_descends_by_reduction} says that the property of being $\mu_{\text{Giv}}$-sub- or superheavy is inherited through reduction, that is if $Z \subset \R P^{8n-1}$ is $\mu_{\text{Giv}}$-sub- or superheavy then its reduction, which is $\chi(Z \cap \Stie_2(\C^{2n})/\Z_2)) \subset V$, is $\mu$-sub- or superheavy. We will describe $L_1,L_2 \subset V$ as reductions of subheavy subsets in $\R P^{8n-1}$, implying that they are $\mu$-subheavy.

The Legendrian $L_1$ is the reduction of the real part $\R P^{4n-1} \subset \R P^{8n-1}$:
$$L_1 = \chi(\R P^{4n-1} \cap \Stie_2(\C^{2n})/\Z_2)\,.$$
Let us compute what it is. Equivalently $L_1$ can be obtained as follows. Consider the diagram
$$S^{8n-1} \supset \Stie_2(\C^{2n}) \xrightarrow{\wh\chi}\Stie_2(\C^{2n})/\SU(2) = V$$
covering the diagram \eqref{eqn:diagram_reduction_RP_to_Grassmannian}. Then $L_1 = \wh\chi(S^{4n-1} \cap \Stie_2(\C^{2n}))$, as is easy to see. The intersection $S^{4n-1} \cap \Stie_2(\C^{2n})$ is just the real part of $\Stie_2(\C^{2n})$, which is the Stiefel variety $\Stie_2(\R^{2n})$ of Euclidean orthonormal real $2$-frames in $\R^{2n}$. The projection $\wh\chi$ identifies points on this variety related by the action of $\SU(2)$. It is easy to see that two points on it are related by $A \in \SU(2)$ if and only if $A$ has real entries, meaning $A \in \SO(2)\subset\SU(2)$, which means that $L_1 = \Stie_2(\R^{2n})/\SO(2) = \Gr_2^+(\R^{2n})$, the Grassmannian of oriented real $2$-planes in $\R^{2n}$.

In order to construct $L_2$, we need to recall the map $b \fc \R^{4n} \to \C^{4n}$, see equation \eqref{eqn:the_map_b}. Identifying $\R^{4n} = \HH^n$ and $\C^{4n} = \HH^n \times \HH^n$, it is given by $b(q) = \frac 1 {\sqrt{2}}(q,jq)$. The group $\SU(2)$, identified with the unit sphere in $\HH$, acts on $\HH^n$ and $\HH^n\times\HH^n$ by multiplication on the right, and $b$ is equivariant with respect to this action. Note that this action on $\HH^n \times \HH^n$ and the action on $(\C^{2n})^2$, viewed as the set of complex matrices of order $2n\times 2$, are the same. Since the action of $\SU(2)$ on $\HH^n = \R^{4n}$ preserves the sphere $S^{4n-1}$, it preserves the image $b(S^{4n-1})\subset \HH^n \times \HH^n$. We then have
$$L_2 = \chi\big(b(S^{4n-1})/\Z_2\cap \Stie_2(\C^{2n})/\Z_2)\big)\,.$$
Alternatively, $L_2 = \wh\chi(b(S^{4n-1})\cap \Stie_2(\C^{2n}))$. Above we showed that $b(S^{4n-1}) \subset \Stie_2(\C^{2n})$. Since $\wh\chi$ is the quotient by $\SU(2)$, we see that $L_2 = b(S^{4n-1})/\SU(2)$, which, since $b$ is $\SU(2)$-equivariant, is diffeomorphic to $S^{4n-1}/\SU(2)$. The action of $\SU(2)$ on $S^{4n-1}$ is precisely the one appearing in the higher Hopf fibration
$$S^{4n-1}\xrightarrow{\SU(2)}\HH P^{n-1}\,,$$
which means that $L_2$ is diffeomorphic to the quaternionic projective space $\HH P^{n-1}$.

The real Legendrian $\R P^{4n-1} \subset \R P^{8n-1}$ is $\mu_{\text{Giv}}$-subheavy by \cite[Lemma 1.23]{Borman_Zapolsky_Qms_Cont_rigidity}. Let us show that $b(S^{4n-1})/\Z_2$ is also subheavy. Recall the unitary map $b_\C$, see \eqref{eqn:unitary_extension_of_b}. Let us denote by $\ol B$ the map it induces on $\R P^{8n-1}$. Since $\ol B$ is induced by a unitary map, it is a contactomorphism belonging to $\Cont_0(\R P^{8n-1})$. Moreover, $b(S^{4n-1})/\Z_2 = \ol B(\R P^{4n-1})$, and since subheaviness is preserved by elements of $\Cont_0$, we see that $b(S^{4n-1})/\Z_2$ is likewise $\mu_{\text{Giv}}$-subheavy. It follows that $L_1,L_2$ are subheavy.

We can now prove Theorem \ref{thm:rigidity_result}. Items (i-iii): that $L_i$ is Legendrian follows from the fact that contact reduction preserves the Legendrian property, in analogy with the fact that symplectic reduction preserves the Lagrangian property. One checks that $K_1=\pi(L_1)$ is $\Gr_2(\R^{2n})$ and that $\pi|_{L_1}$ is the canonical double cover, and that $\pi|_{L_2}\fc L_2 \to K_2$ is a diffeomorphism. Since $\pi|_{L_i} \fc L_i \to K_i$ are covering maps and $\pi$ is the projection of a prequantization bundle, it follows that $K_i$ is Lagrangian.

Since $L_i$ is $\mu$-subheavy, $K_i = \pi(L_i)$ is $\zeta$-superheavy, as are its Hamiltonian images, and thanks to Proposition \ref{prop:relation_superheavy_preq_bdle_total_sp_base} and the fact that the collection of $\zeta$-superheavy sets is invariant by the action of the Hamiltonian group, $K_i$ must intersect any Hamiltonian image of $K_j$, proving item (iv). For item (v) we note that $Q_i$ is $\mu$-subheavy, since it contains $L_i$, and since it is invariant by the Reeb flow, \cite[Proposition 1.13]{Borman_Zapolsky_Qms_Cont_rigidity} implies that it is superheavy. Therefore $Q_i$ must intersect $\psi(L_j)$ for any $\psi \in \Cont_0(V,\xi)$ by Proposition \ref{prop:properties_rigid_subsets_ct_mfds}. The proof of Theorem \ref{thm:rigidity_result} is complete.

\bibliography{biblio}
\bibliographystyle{plain}

\end{document}